\documentclass[preprint, 12pt, sort&compress]{elsarticle}

\usepackage{latexsym}
\usepackage{amsfonts}
\usepackage{graphicx,amssymb,bm}
\usepackage{amsmath}
\usepackage{amssymb}
\usepackage[mathscr]{eucal}
\usepackage{enumerate}
\usepackage{amsthm}

\begin{document}

\newtheorem{tm}{Theorem}[section]
\newtheorem{pp}{Proposition}[section]
\newtheorem{lm}{Lemma}[section]
\newtheorem{df}{Definition}[section]
\newtheorem{tl}{Corollary}[section]
\newtheorem{re}{Remark}[section]
\newtheorem{eap}{Example}[section]

\newcommand{\pof}{\noindent {\bf Proof} }
\newcommand{\ep}{$\quad \Box$}

\newcommand{\al}{\alpha}
\newcommand{\be}{\beta}
\newcommand{\var}{\varepsilon}
\newcommand{\la}{\lambda}
\newcommand{\de}{\delta}
\newcommand{\st}{\stackrel}

\allowdisplaybreaks

\begin{frontmatter}

\title{Some notes on approximately greater than relations on fuzzy sets
 }

\author{Huan Huang}
\ead{hhuangjy@126.com}
\address{Department of Mathematics, Jimei
University, Xiamen 361021, China}

\begin{abstract}

We consider approximately greater than relations on fuzzy sets and discuss their properties.

\end{abstract}


\end{frontmatter}

\section{Introduction and preliminaries}

We use $F(\mathbb{R}^m)$ to represent all
fuzzy sets on $\mathbb{R}^m$, i.e. functions from $\mathbb{R}^m$
to $[0,1]$.
$
\mathbf{2}^{\mathbb{R}^m }:=  \{  S: S\subseteq \mathbb{R}^m      \}
$
can be embedded in $F (\mathbb{R}^m)$, as any $S \subset \mathbb{R}^m$ can be
seen as its characteristic function, i.e. the fuzzy set
\[
\widehat{S}(x)=\left\{
\begin{array}{ll}
1,x\in S, \\
0,x\notin S.
\end{array}
\right.
\]
If there is no confusion,
we shall simply write
$S$ to denote $\widehat{S}$.
Let $r\in \mathbb{R}$.
If there is no confusion,
we will use
$\mathbf{r}$ to denote $\mathbf{r}_\mathbf{m}=(r,r,\ldots, r)$ in $\mathbb{R}^m$,
the singleton
 $\{\mathbf{r}_\mathbf{m}\}$, or
the fuzzy set $\widehat{    \{\mathbf{r}_\mathbf{m}\}    }$.

Suppose
$u\in F(\mathbb{R}^m)$, we use $[u]_{\al}$ to denote the $\al$-cut of
$u$, i.e.
\[
[u]_{\al}=\begin{cases}
\{x\in \mathbb{R}^m : u(x)\geq \al \}, & \ \al\in(0,1],
\\
{\rm supp}\, u=\overline{\{x \in \mathbb{R}^m: u(x)>0\}}, & \ \al=0.
\end{cases}
\]
In the sequel
of
this paper, our discussion involves several kinds
of
fuzzy sets such as $F_{USCG}   (\mathbb{R}^m)$, $E^m_{nc}$, $E^m$, $L(E^m)$ which    are listed in the following.
\begin{itemize}
  \item
We say a fuzzy set
  $u\in F_{USCG} (\mathbb{R}^m)$
if and only if
$u\in F(\mathbb{R}^m)$ and
$[u]_\al$ is a compact set in $\mathbb{R}^m$ for each $\al\in(0,1]$.
$F_{USCG} (\mathbb{R}^m)$ is a subset of upper semi-continuous fuzzy sets on $\mathbb{R}^m$.

\item We say a fuzzy set
  $u\in E^m_{nc}$
if and only if
$u\in F(\mathbb{R}^m)$ and
$[u]_\al$ is a nonempty compact convex set in $\mathbb{R}^m$ for each $\al\in(0,1]$.
A fuzzy set in $E^m_{nc}$ is called a (noncompact) fuzzy number.

\item  We say a fuzzy set
  $u\in E^m$
if and only if
$u\in F(\mathbb{R}^m)$ and
$[u]_\al$ is a nonempty compact convex set in $\mathbb{R}^m$ for each $\al\in [0,1]$.
A fuzzy set in $E^m$ is called a (compact) fuzzy number.

  \item
  We say
 a fuzzy set $u \in L(E^m)$
if and only if
$u \in F(\mathbb{R}^m)$ and $[u]_\al$ is a cell, i.e., $[u]_\al= \Pi_{k=1}^m [{u_k}^-(\al), {u_k}^+(\al)]$ for all $\al\in [0,1]$,
where ${u_k}^-(\al), {u_k}^+(\al) \in \mathbb{R}$ and ${u_k}^-(\al) \leq {u_k}^+(\al)$.
A fuzzy set in $L(E^m)$ is called an $m$-cell fuzzy number.
\end{itemize}

The concept of  $m$-cell fuzzy number was introduced by
Wang and Wu \cite{wangx}. It can be checked
that
$L(E^m)   \subsetneq E^m \subsetneq E^m_{nc} \subsetneq  F_{USCG}   (\mathbb{R}^m) $.

Buckley \cite{buckley} has introduced the approximately greater than relation on fuzzy sets in $E^1_{nc}$.

\begin{df} \label{buckley} \cite{buckley}\
  Let $u,v \in E^1_{nc}$, the set of $1$-dimensional noncompact fuzzy numbers and $\theta>0$.
  Define
   $$ G(u\geq v):= \sup_{x\geq y } (   u(x) \wedge   v(y)   ).$$
    It is called
    \begin{itemize}
      \item  $u >_\theta v$, if $G(u \geq v) =1$ and $ G(v \geq u) <  \theta$.
      \item  $u=_\theta v$, if  $G(u \geq v) \geq \theta$ and $ G(v \geq u) \geq \theta$.
      \item  $u \geq_\theta v$, if  $u >_\theta v$ or   $u=_\theta v$.
    \end{itemize}
\end{df}


\begin{re}
Wang (the Corollary after Lemma 1 in \cite{wang}) have pointed out
that, for any
  $u,v\in L(E^1)=E^1$,
 $ u \geq_\theta v  \Longleftrightarrow  G(u \geq v) \geq \theta$.

\end{re}

\section{Extension of $G$}

The concept of $ G(u\geq v)$   can be easily extended.
  We can define
   $$ \bm{G(u \bm{ \succeq } v)}:= \sup_{x \succeq y } (   u(x) \wedge   v(y)   )$$
  where
   $u,v \in F(\mathbb{R}^m)$ and $\succeq$ is a partial order on $\mathbb{R}^m$.

\begin{lm}
\label{lac}
Suppose $u,v \in F(\mathbb{R}^m)$ and $\theta>0$.
Then
\\
(\romannumeral1) \
  $G(u\succeq v)<\theta $ if and only if there exists an $\alpha \in [0, \theta)$ such that $G([u]_\al \succeq [v]_\al) =0$.
  \\
(\romannumeral2) \
  $G(u\succeq v)\geq\theta $ if and only if $G([u]_\al \succeq [v]_\al) =1$ for all $\alpha\in [0, \theta)$.
\end{lm}

\begin{proof}
  (\romannumeral1) is equivalent to (\romannumeral2). The proof of this lemma is easy.
\end{proof}

\begin{itemize}
  \item A partial order $\succeq$ on $\mathbb{R}^m$ is said to be \textbf{preserved by limit operation}, if $x \succeq y$
when $x=\lim_{n\to \infty} x_n$, $y=\lim_{n\to \infty} y_n$
and
$x_n \succeq y_n$, $n=1,2,\ldots$.
\end{itemize}

\begin{lm}\label{pln} Suppose that $u,v \in   F_{USCG}(\mathbb{R}^m)$, $\theta>0$, and the partial order $\succeq$ is preserved by limit operation. Then
  $G(u\succeq v)\geq \theta $ if and only if $G([u]_\theta \succeq [v]_\theta) =1$.
\end{lm}

\begin{proof}
  The proof is easy.
\end{proof}

\section{ $\succeq_\theta$ --- a generation of $\geq_\theta$ }

\begin{df} \label{newaeu}
  Suppose that $u,v \in F(\mathbb{R}^m)$, $\theta>0$ and $\succeq$ is a partial order on $\mathbb{R}^m$. We say
\bm{$u \succeq_\theta v$} if and only if $G(u\succeq v)\geq \theta $.
\end{df}

Let $\succeq^a$ be a partial order on $\mathbb{R}^m$
     and let $\succeq^b$ be a partial order on $\mathbb{R}^n$.
We say $f: (\mathbb{R}^m, \succeq^a)  \to (\mathbb{R}^n, \succeq^b)$ is nondecreasing if $x \succeq^a y $ in $\mathbb{R}^m$
implies $f(x) \succeq^b f(y) $ in $\mathbb{R}^n$.

Given $f: \mathbb{R}^m  \to \mathbb{R}^n$ and $u\in F(\mathbb{R}^m)$.
Then
 $f(u)$ is defined via the Zadeh's extension principle.

\begin{tm} \label{mon}
   Suppose that $u,v \in F(\mathbb{R}^m)$, $\theta>0$ and that $f: (\mathbb{R}^m, \succeq^a)  \to (\mathbb{R}^n, \succeq^b)$ is nondecreasing.
  Then $u \succeq^a_\theta v$ implies $f(u)\succeq^b_\theta f(v)$.
\end{tm}

\begin{proof}
Assume $u \succeq^a_\theta v$. From Lemma \ref{lac}, we know
that $G([u]_\xi \succeq^a [v]_\xi) =1$
for all $\xi\in [0, \theta)$, and
hence
$G([f(u)]_\xi \succeq^b [f(v)]_\xi) =1$
for all $\xi\in [0, \theta)$.
Thus
$f(u)\succeq^b_\theta f(v)$.
\end{proof}

Suppose that $x=(x_1,\ldots,x_m),y=(y_1,\ldots, y_m) \in \mathbb{R}^m$. We call \bm{$x \succeq^1_m y$} if $x_i \geq y_i$ for $i=1,2,\ldots, m$.
If there is no confusion,
we
will write \bm{$ \succeq^1 $} to denote $\succeq^1_m$ for simplicity.
It can be checked that $\succeq^1 $ is a partial order on $\mathbb{R}^m$, which is
preserved by limit operation.

\begin{re}
  We can construct
  $f: (\mathbb{R}^m, \succeq^1)  \to (\mathbb{R}^n, \succeq^1)$, which is nondecreasing
  and
  $f(u) \in L(E^n)$ when $u \in L(E^m)$. From Theorem \ref{mon}, it is clear
  that
 $u \succeq^1_\theta v$ implies $f(u)\succeq^1_\theta f(v)$ for  $u,v \in L(E^m)$ and $\theta>0$.

 Define $f: (\mathbb{R}^m, \succeq^1)  \to (\mathbb{R}^m, \succeq^1)$
  as follows:
  $f(x_1,\ldots, x_m) = (f_1(x_1),\ldots, f_m(x_m))$, where $f_i:\mathbb{R} \to \mathbb{R} $, $i=1,2,\ldots,m$.
 Then
   $f$ is nondecreasing iff  $f_i$ is nondecreasing, i.e. $f_i(x) \geq f_i(y)$ when $x \geq y$, $i=1,2,\ldots,m$,
   and
  $f$
  is continuous iff $f_i$ is continuous, $i=1,2,\ldots,m$.

  Suppose that $f$ defined above
  is    continuous.
  Then      $f(u) \in L(E^m)$ when $u\in L(E^m)$,
  and
for any $\theta \in [0,1]$,
      \begin{align*}
      [f(u)]_\theta & = f([u]_\theta)\\
      &        =f(\Pi_{i=1}^m    [u_i^-(\theta), u_i^+(\theta)]) \\
      & =  \Pi_{i=1}^m   f_i[ u_i^-(\theta) ,    u_i^+(\theta)  ].
    \end{align*}
Moreover, if $f$ is nondecreasing and continuous,
then for any $\theta \in [0,1]$,
  $$ [f(u)]_\theta= \Pi_{i=1}^m   [ f_i( u_i^-(\theta) ),   f_i(   u_i^+(\theta) ) ].$$

In this way, we can  construct $f: (\mathbb{R}^m, \succeq^1)  \to (\mathbb{R}^n, \succeq^1)$, $n\leq m$, which is   nondecreasing,   continuous and
$f(u) \in L(E^n)$ when $u\in L(E^m)$.

For example, let
$f(x_1, x_2, x_3, x_4)=(g(x_4), h(x_2, x_1), k(x_3))$, where $g:\mathbb{R} \to \mathbb{R}$,
$h:(\mathbb{R}^2, \succeq^1) \to \mathbb{R}$ and $k:\mathbb{R} \to \mathbb{R}$.
Then
   $f:  (\mathbb{R}^4, \succeq^1) \to (\mathbb{R}^3, \succeq^1)$ is nondecreasing iff  $g,h,k$ is nondecreasing,
   and
  $f$
  is continuous iff $g,h,k$  is continuous.
If
$f$ is continuous, then $f(u) \in L(E^3)$ when $u\in L(E^4)$,
and for any $\theta \in [0,1]$,
\begin{align*}
&[f(u)]_\theta = f([u]_\theta)\\
&= g[ u_4^-(\theta),     u_4^+(\theta)]\times h([u_2^-(\theta), u_2^+(\theta)] \times [u_1^-(\theta), u_1^+(\theta)])
\times k[u_3^-(\theta),  u_3^+(\theta) ].
\end{align*}
If
$f$ is nondecreasing and  continuous,
then for any $\theta \in [0,1]$,
$$[f(u)]_\theta =  [ g( u_4^-(\theta) ),   g(   u_4^+(\theta) ) ]
       \times [h(u_2^-(\theta), u_1^-(\theta)),  h(u_2^+(\theta), u_1^+(\theta))  ]
       \times [k( u_3^-(\theta) ), k( u_3^+(\theta) )].$$

\end{re}

\begin{tm} \label{uscgce}
   Suppose that $u,v \in   F_{USCG}(\mathbb{R}^m)$, $\al>0$. Then
   $u \succeq^1_\alpha v$
if and only if $G([u]_\al \succeq^1 [v]_\al) =1$.
\end{tm}

\begin{proof}
  The desired result follows immediately from Lemma \ref{pln}.
\end{proof}

\begin{tm} \label{chm}
   Suppose that $u,v \in   L(E^m)$, $\al>0$. Then the following statements are equivalent:
   \\
   (\romannumeral1) \ $u \succeq^1_\alpha v$.
   \\
   (\romannumeral2) \   $G(u \succeq^1 v) \geq \al$.
   \\
   (\romannumeral3) \   $G([u]_\al \succeq^1 [v]_\al) =1$.
   \\
      (\romannumeral4) \  ${u_k}^+(\al) \geq {v_k}^-(\al) $ for $k=1,2,\ldots,m$.
\end{tm}

\begin{proof} The equivalence of statements  (\romannumeral1), (\romannumeral2) and (\romannumeral3)
follow
  from Definition \ref{newaeu} and Theorem \ref{uscgce}.

  Assume statement  (\romannumeral3) holds. Then there exists $x \in [u]_\al $ and $y \in [v]_\al$
  such that
  $x\succeq^1 y$. This follows immediately that ${u_k}^+(\al) \geq {v_k}^-(\al) $ for $k=1,2,\ldots,m$.
  Thus (\romannumeral3) implies (\romannumeral4).

  It remains to prove
  that
  (\romannumeral4) implies (\romannumeral3).
  Accordingly,
  let   ${u_k}^+(\al) \geq {v_k}^-(\al) $ for $k=1,2,\ldots,m$.
  Then
  $x :=({u_1}^+(\al),\ldots, {u_k}^+(\al),\ldots, {u_m}^+(\al))  \succeq^1 y:= ({v_1}^-(\al),\ldots, {v_k}^-(\al),\ldots, {v_m}^-(\al))$;
  whence
     $G([u]_\al \succeq^1 [v]_\al) =1$, and therefore (\romannumeral4) implies (\romannumeral3).
\end{proof}

\begin{re}

 $u \succ^1_\al v$ can be defined as $u \succeq^1_\al  v$ but $v \not\succeq^1_\al u$.
In this sense, for $u,v \in L(E^m)$,
$u  \succ^1_\al  v$
iff
${u_k}^+(\al) \geq {v_k}^-(\al) $ for $k=1,2,\ldots,m$
and
${u_j}^-(\al) > {v_j}^+(\al)$ for some $j$ with $1\leq j \leq m$.

Sometimes, for $u,v \in L(E^m)$, $u  \succ^1_\al  v$ is defined
as
${u_k}^+(\al) \geq {v_k}^-(\al) $ for $k=1,2,\ldots,m$
and
${u_j}^+(\al) > {v_j}^-(\al)$ for some $j$ with $1\leq j \leq m$.

 Suppose that $m \geq 2$.
It is easy to check
that, for any definition of ``$\succ^1_\al$'' on $L(E^m)$ mentioned above,
there
exists  $u$, $v$, $u_n$, $v_n$, $n=1,2,\ldots$, in   $L(E^m)$
such that
 $\{u_n\}$ level converges to $u$,  $\{v_n\}$ level converges to $v$, 
$u  \succ^1_\al v$,
but
$u_n \succeq^1_\al v_n$ does not hold for any $n=1,2,\ldots$.

\end{re}

$u\in F(\mathbb{R}^m)$ is said to
be
\textbf{positive} if for any $x=(x_i)$ with $u(x)>0$, it holds that $x_i >0$, $i=1,2,\ldots, m$.
$u\in F(\mathbb{R}^m)$ is said to
be
\textbf{negative} if $-u$ is positive.

$u\in F(\mathbb{R}^m)$ is said to
be
\textbf{strong positive} if for any $x=(x_i)\in [u]_0$, it holds that $x_i >0$, $i=1,2,\ldots, m$.
$u\in F(\mathbb{R}^m)$ is said to
be
\textbf{strong negative} if $-u$ is strong positive.

The four arithmetic operators $+,-,\times,\div$ on $F(\mathbb{R}^m)$ are defined as usual, which are based on the Zadeh's extension principle.
See \cite{wangx}.

We can consider properties of $\succeq^1_\al$. The following theorem gives one of them.

\begin{tm}
Suppose that  $u,v \in   F(\mathbb{R}^m)$, $\al>0$.
If $u \succeq^1_\al v$ and $v$ is positive,
then $u/v  \succeq^1_\al \mathbf{1}$.
\end{tm}

\begin{proof}
  Assume $u \succeq^1_\al v$. Then from Theorem \ref{chm},
  we
  know that there exists
  $\{x_n\}$, $\{y_n\}$
  such that
   $x_n\in  [u]_{\frac{n-1}{n}\al}$, $y_n \in [v]_{\frac{n-1}{n}\al}$
and
  $x_n \succeq^1 y_n$. Since $v$ is positive, then $x_n/y_n \succeq^1 \mathbf{1}$.
Thus
   $u/v  \succeq^1_\al \mathbf{1}$.
\end{proof}

\begin{re}
If $u,v \in L(E^m)$ and $v$ is strong positive or   strong negative, then $u/v$ is in $L(E^m)$,
and
for each $\theta \in [0,1]$,
$$\left[  \frac{u}{v} \right]_\theta = \frac  {[u]_\theta}   { [v]_\theta}
=  \frac{ \prod_{i=1}^m  [u_i^-(\theta),  u_i^+(\theta)]   }   {  \prod_{i=1}^m  [v_i^-(\theta),  v_i^+(\theta)]  }
=   \prod_{i=1}^m     \frac{  [u_i^-(\theta),  u_i^+(\theta)]  }   { [v_i^-(\theta),  v_i^+(\theta)]}.    $$
  \end{re}

\section{ $\succeq'_\theta$ --- another generation of $\geq_\theta$}

The following $\succeq'_\theta$
is another
generation
of  $\geq_\theta$
given in Definition \ref{buckley}.

\begin{df} \label{buckleyg}\
Suppose that $u,v \in \bm{F'(\mathbb{R}^m)}:= \{w: w\in F(\mathbb{R}^m) \ \mbox{and} \ [w]_1 \not= \emptyset\}$, $\succeq$ is a partial order on $\mathbb{R}^m$ and $\theta>0$.
 It is said
 that
    \begin{itemize}
      \item  $u \succ'_\theta   v$, if $G(u \succeq v) =1$ and $ G(v \succeq u) <  \theta$.
      \item  $u='_\theta v$, if  $G(u \succeq v) \geq \theta$ and $ G(v \succeq u) \geq \theta$.
      \item  $u \succeq'_\theta v$, if  $u \succ'_\theta v$ or   $u='_\theta v$.
    \end{itemize}
\end{df}

Compare $\succeq'_\theta$ given in Definition \ref{buckleyg}
and
$\succeq_\theta$
in
Definition \ref{newaeu}.

\begin{tm} \label{prm}
  Suppose that $u,v \in F'(\mathbb{R}^m)$, $\succeq$ is a partial order on $\mathbb{R}^m$ and $\theta>0$.
  Then
  $u \succeq'_\theta v$ implies
   $u \succeq_\theta v$.
\end{tm}

\begin{proof}
  The desire result follows immediately from Definitions \ref{newaeu} and \ref{buckleyg}.
\end{proof}

\begin{lm}\label{lrn}
  Suppose that $u,v \in F'(\mathbb{R}^1)$ and $\theta>0$.
Then
$ G(v \succeq^1 u) <  \theta$
implies
  $G(u \succeq^1 v) =1$.
\end{lm}

\begin{proof}
  The desired result follows immediately from  Lemma \ref{lac}.
\end{proof}

\begin{tm} \label{nde}
  Suppose that $u,v \in F'(\mathbb{R}^1)$ and $\theta>0$. Then
  $u {\succeq^1_\theta}' v$
  if and only if
   $u \succeq^1_\theta v$.
\end{tm}

\begin{proof}
  From Theorem \ref{prm}, we only need to show that    $u \succeq^1_\theta v$
implies
   $u {\succeq^1_\theta}' v$,
which can be deduced
 directly
   from
   Lemma \ref{lrn}.
\end{proof}

\begin{re}
 Let $m\geq2$.
The conclusion in Theorem \ref{nde} may not hold when $u,v$ are fuzzy sets in $F'(\mathbb{R}^m)$.
  In fact,
  it is easy to construct a pair of $u,v \in L(E^m)$
such that   $u \succeq^1_\theta v$ but   $u {\not\succeq^1_\theta}' v$.

\end{re}

\end{document}